\documentclass{amsart}
\usepackage{amscd, amssymb, portland, calc, ifthen}
\usepackage[matrix,arrow,curve,cmtip]{xy}

\usepackage{eucal}
\usepackage{mathrsfs}
\DeclareMathAlphabet{\mathpzc}{OT1}{pzc}{m}{it}
\let \: \colon

\newcommand{\vbl}{-}

\newcommand{\tn}{\otimes}           % Tensor

\newcommand{\mathbold}{\bf}

\newcommand{\mc}[1]{{\mathcal {#1}}}

\newcommand{\longmap}{{\,\longrightarrow\,}}
\newcommand{\Hom}{{\mathrm{Hom}}}
\newcommand{\Map}{{\mathrm{Map}}}

\newcommand{\Gal}{{\mathrm{Gal}}}

\newcommand{\bF}{{\mathbold F}}
\newcommand{\bQ}{{\mathbold Q}}

\newcommand{\bZ}{{\mathbold Z}}

\newcommand{\bN}{{\mathbold N}}

\newcommand{\p}{\mathfrak{p}}

\newcommand{\comment}[1]{}

\renewcommand{\leq}{\leqslant}
\renewcommand{\geq}{\geqslant}
\newcommand{\sub}{\subseteq}

\newcommand{\zhat}{\hat{\bZ}}

\newcommand{\bm}[1]{{(\bZ/{#1}\bZ)}^{\circ}}
\newcommand{\bg}[1]{{(\bZ/{#1}\bZ)}^{*}}
\newcommand{\zpos}{\bN'}
\newcommand{\dvd}{\mid}

\newtheoremstyle{mythm}{}{}%
  {\itshape}%   Body font
  {}%           Indent amount (empty = no indent, \parindent = para indent)
  {\bfseries}%  Thm head font
  {}%           Punctuation after thm head
  { }%          Space after thm head (\newline = linebreak)
  {\thmnumber{#2.\hspace{1.5mm}}\thmname{#1}\thmnote{ #3}.}%  Thm head spec

\newtheoremstyle{myrmk}{}{}%
  {}%           Body font
  {}%           Indent amount (empty = no indent, \parindent = para indent)
  {}%   Thm head font
  {}%           Punctuation after thm head
  { }%          Space after thm head (\newline = linebreak)
  {{\bfseries\thmnumber{#2.\hspace{1.5mm}}}{\itshape\thmname{#1}}\thmnote{ #3}.}%  Thm head spec

\numberwithin{equation}{subsection}

\theoremstyle{mythm}
\newtheorem{thm}[subsection]{Theorem}
\newtheorem{prop}[subsection]{Proposition}
\newtheorem{lemma}[subsection]{Lemma}
\newtheorem{cor}[subsection]{Corollary}

\theoremstyle{myrmk}

\theoremstyle{plain}
\newtheorem*{thm*}{Theorem}
\newtheorem*{prop*}{Proposition}
\newtheorem*{cor*}{Corollary}
\newtheorem*{conj*}{Conjecture}

\theoremstyle{definition}

\makeatletter
\def\@seccntformat#1{\@ifundefined{#1@cntformat}%
{\csname the#1\endcsname\quad}%  default
{\csname #1@cntformat\endcsname}% individual control
}
\def\section@cntformat{\thesection.\enspace}
\def\subsection@cntformat{\thesubsection.}
\makeatother

\newcommand\mnote[1]{}
\newcommand\lb[1]{\label{#1}\mnote{#1}}

\newcounter{hour}\newcounter{minute}
\newcommand{\printtime}{\setcounter{hour}{\time/60}%
        \setcounter{minute}{\time-\value{hour}*60}%
        \ifthenelse{\value{hour}<10}{0}{}\thehour:%
        \ifthenelse{\value{minute}<10}{0}{}\theminute}

\begin{document}

%\title{Maximal $\Lambda$-orders over $\bZ$ \\Preliminary report}
\title{Galois theory and integral models of $\Lambda$-rings}
\author[J.~Borger, B.~de~Smit]{James Borger, Bart de Smit}
\address{James Borger\\
Mathematical Sciences Institute\\
Building 27\\
Australian National University\\
ACT 0200\\
Australia}

\address{Bart de Smit\\
Mathematisch Instituut\\
Universiteit Leiden\\
Postbus 9512\\
2300 RA Leiden\\
The Netherlands 
}
%\curraddr{}
\date{\today. \printtime}
\email{borger@maths.anu.edu.au, desmit@math.leidenuniv.nl}
%\thanks{{\em Keywords:} }
\thanks{ 
(secondary)}
% 13K05 - witt vectors and related rings
% 11R37 - class field theory
% 19L20 - $J$-homomorphism, Adams operations
% 16W99 - rings and algebras with additional structure
%\thanks{THANKS}

\begin{abstract}
We show that any $\Lambda$-ring, in the sense of Riemann--Roch theory, which is 
finite \'etale
over the rational numbers and has an integral model as a $\Lambda$-ring is 
contained in a product of cyclotomic fields.  In fact, we show that the category of 
them is described in a Galois-theoretic way in terms of the monoid of pro-finite 
integers under multiplication and the cyclotomic character.  We also study the 
maximality of these integral models and give a more precise, integral version of 
the result above.
These results reveal an interesting relation between $\Lambda$-rings and
class field theory.

%We show that the category of $\Lambda$-rings that are finite \'etale
%over the rational numbers and have integral models as $\Lambda$-rings
%is naturally anti-equivalent
%to the category of finite
%% discrete
%sets equipped with an
%% continuous
%action of the monoid 
%%$\zhat^{\circ}$ 
%of pro-finite
%integers under multiplication. 
%The proof uses class field theory over the rational numbers.
%We also study the maximality of these integral models, and prove that
%every nilpotent-free $\Lambda$-ring of finite rank is a sub-$\Lambda$-ring of a power of the coordinate ring of the integral $n$-torsion on the multiplicative group.
\end{abstract}

\maketitle
\section*{Introduction}

According to the most common definition, a $\Lambda$-ring structure
on a commutative ring $R$ is a
sequence of set maps $\lambda_1,\lambda_2,\dots$ from $R$ to itself
that satisfy certain complex implicitly stated axioms.
This notion was introduced by Gro\-then\-dieck~\cite{Grothendieck:Chern},
under the name special $\lambda$-ring, to give an
abstract setting for studying the structure on Gro\-then\-dieck
groups inherited from exterior power operations; and as far as we
are aware, with just one exception~\cite{Clauwens:Lambda}, $\Lambda$-rings
have been studied in the literature for this purpose only.

However, it seems that the study of abstract
$\Lambda$-rings---those having no apparent relation to
$K$-theory---will have something to say about number theory.
%, and the first
%author is in the midst of a long project exploring a general
%philosophy about this.  
%For these
%applications, it is important to consider
%$\Lambda$-rings that are not finitely generated as rings.  But then, for
%just this reason, 
One example of such a relationship
is the likely existence of strong arithmetic
restrictions on the complexity
of finitely generated rings that admit a $\Lambda$-ring structure.
%can be then becomes interesting from a number theoretic perspective.
The primary purpose of this paper is to investigate this issue
in the zero-dimensional case.  Precisely,
what finite
\'etale $\Lambda$-rings over $\bQ$ are of the form $\bQ\tn A$, where $A$
is a $\Lambda$-ring that is finite flat over $\bZ$?  

%what
%remains of the basic theory of rings of integers in finite \'etale
%$\bQ$-algebras when $\Lambda$-actions are imposed on everything?

We will consider $\Lambda$-actions only on rings whose underlying
abelian group is torsion free, and giving a
$\Lambda$-action on such a ring $R$
is the same as giving commuting ring endomorphisms
$\psi_p\colon R\to R$, one for each prime $p$,
lifting the Frobenius map modulo $p$---that is, such that
$\psi_p(x)-x^p\in pR$ for all $x\in R$.
(The equivalence of this with Grothendieck's original
definition is proved in Wilkerson~\cite{Wilkerson}.)
An example that will be important here is
$\bZ[\mu_r]=\bZ[z]/(z^r-1)$, where $r$ is a positive integer and
$\psi_p$ sends $z$ to $z^p$.
A morphism of torsion-free $\Lambda$-rings
is the same as a ring map $f$ that satisfies
$f\circ\psi_p = \psi_p\circ f$ for all
primes $p$.

Note that if $R$ is a $\bQ$-algebra,
the congruence conditions in the definition 
above disappear.
Also, Galois theory as interpreted by Grothendieck gives 
an anti-equivalence between the category of
finite \'etale $\bQ$-algebras and the category of finite
discrete sets equipped
with a continuous action of the absolute Galois group $G_{\bQ}$ with
respect to a fixed algebraic closure $\bar{\bQ}$.
Combining these two remarks, we see that the category of
$\Lambda$-rings that are finite \'etale over $\bQ$ is nothing more
than the category of finite discrete sets equipped with a continuous
action of the monoid $G_{\bQ}\times\zpos$, where $\zpos$ is the 
monoid $\{1,2,\dots\}$ under multiplication with the discrete topology.
This is because $\zpos$ is freely generated as a commutative 
monoid by the prime numbers.

It is not always true, however, that such a $\Lambda$-ring $K$ has an
integral $\Lambda$-model, by which we mean a sub-$\Lambda$-ring $A$,
finite over $\bZ$, such that $\bQ\tn A = K$.
In order to formulate exactly when this happens, we write $\zhat^{\circ}$ for the
set of profinite integers viewed as a topological monoid under multiplication,
and we consider the continuous monoid map
\[
G_{\bQ}\times\zpos \longmap \zhat^{\circ}
\]
given by the cyclotomic character $G_{\bQ}\to\zhat^*\subset \zhat^{\circ}$ on the first
factor and the natural inclusion on the second. Note that this
map has a dense image.
Now let $K$ be a finite \'etale algebra over $\bQ$, 
and let $S$ be the set of ring maps from $K$ to $\bar{\bQ}$.
Suppose $K$ is endowed with a $\Lambda$-ring structure, so
that we have an induced monoid map 
\[
G_{\bQ}\times\zpos \longmap \Map(S,S),
\]
where $\Map(S,S)$ is the monoid of set maps from $S$ to itself.
With the discrete topology on $\Map(S,S)$ this map is continuous.

\begin{thm}
\lb{thm-A}
The $\Lambda$-ring $K$ has an integral $\Lambda$-model if and only if
the action of $G_{\bQ}\times\zpos$ on $S$ factors (necessarily uniquely) through $\zhat^{\circ}$; in more precise terms, if and only if 
there is a 
%(necessarily unique) 
continuous monoid map $\zhat^{\circ}\to \Map(S,S)$
so that the diagram
\[
\xymatrix{
G_{\bQ}\times\zpos 
\ar[r] \ar[dr] &
\zhat^{\circ} \ar[d]\\
&\Map(S,S)}
\]
commutes.
\end{thm}

It follows that the category of such $\Lambda$-rings is anti-equivalent to
the category of finite discrete sets with a continuous action of
$\zhat^{\circ}$ and that
every such $\Lambda$-ring 
%each one
is contained in a product of cyclotomic fields.
It also suggests
there is
an interesting theory of a $\Lambda$-algebraic fundamental monoid,
analogous to that of %Grothendieck's
the usual algebraic fundamental group,
%rather than algebraic fundamental group,
but we will leave this for a later date.
%might be a general theory of
%a monoid that is the $\Lambda$-ring analogue of Grothendieck's
%algebraic fundamental group,
%there is an interesting theory of
%the $\Lambda$-algebraic fundamental monoid,
%for $\Lambda$-rings
%involving monoids instead of just groups,
%, rather
%than group, 

Another consequence is that the elements $-1,0\in\zhat^\circ$ give an involution 
$\psi_{-1}$ (complex conjugation) and a idempotent endomorphism $\psi_{0}$ on any 
$K$ with an integral $\Lambda$-model.  In $K$-theory, the dual and rank also give such 
operators, but here they come automatically from the $\Lambda$-ring structure.  
Also observe that any element of the subset $0S\subseteq S$ is Galois invariant and hence 
corresponds to a direct factor $\bQ$ of the algebra $K$.
%Also observe that the image of 
%$\psi_0$ is Galois invariant and is hence equal to $\bQ^n$ for some $n$.  
Therefore $K$ cannot be a field unless $K=\bQ$.

In the first section, we give some basic facts and also show the
sufficiency of the condition in the theorem above.  
%The proof is completely elementary.  
In the second section, we
show the necessity.  The proof combines a simple application of the 
Kronecker--Weber theorem and the Chebotarev density theorem with some elementary 
but slightly intricate work on actions of the monoid $\zhat^\circ$.

The third section gives a proof of the following theorem:
\begin{thm}
\label{thm-B}
The $\Lambda$-ring $\bZ[\mu_r]$ is the maximal integral $\Lambda$-model of $\bQ\tn\bZ[\mu_r]$.
\end{thm}
Of course, the non-$\Lambda$ version of this statement is false---the usual 
maximal order of $\bQ\tn\bZ[\mu_r]$ is a product of rings of integers in 
cyclotomic fields and strictly contains $\bZ[\mu_r]$, if $r>1$.

A direct consequence of these theorems is:

\begin{cor}
\label{cor-C}
Every $\Lambda$-ring that has finite rank as an abelian group and 
has no non-zero nilpotent elements is a
sub-$\Lambda$-ring of a $\Lambda$-ring 
of the form $\bZ[\mu_r]^n$.
\end{cor}

We do not need to require that the ring be torsion free because any torsion
element in a $\Lambda$-ring is nilpotent, by an easy lemma
attributed to G.\ Segal~\cite[p.\ 295]{Dress}.
We emphasize that while the definition of $\Lambda$-ring that we gave above
does not literally require the ring to be torsion free, 
it is not the correct definition in the absence of this assumption.
In particular, Segal's lemma and the corollary above are false if the naive
definition is used;
for example, take a finite field.  For the definition of $\Lambda$-ring
in the general case, 
see~\cite{Grothendieck:Chern},~\cite{Wilkerson}, or~\cite{Borger-Wieland:PA}.

% beyond the torsion-free case.

Finally, many of the questions answered in this paper have analogues over general 
number fields.  There one would use Frobenius lifts modulo prime ideals of the ring 
of integers, and then general class field theory and
the class group come in, as well as the theory of complex
multiplication in particular cases.
Because these analogues of $\Lambda$-rings are not objects of prior
interest, we have not included anything about them here.  But it is clear 
that finite-rank 
$\Lambda$-rings, in this generalized sense or the original, are fundamentally objects of 
class field theory and that 
they offer a slightly different perspective on the subject.
It would be interesting to explore this further.

%\begin{cor}
%\label{cor-C}
%Every $\Lambda$-ring that is finitely generated free as an abelian group and 
%has no non-zero nilpotent elements is a
%sub-$\Lambda$-ring of a $\Lambda$-ring 
%of the form $\bZ[\mu_r]^n$.
%\end{cor}
%
%In fact, this result also holds without the freeness condition because any
%torsion element in a $\Lambda$-ring is nilpotent, as explained to us by
%Clauwens. To see this, one must use the definition of $\Lambda$-ring that
%is correct even in the presence of non-zero torsion;
%see~\cite{Grothendieck:Chern},~\cite{Wilkerson}, or~\cite{Borger-Wieland:PA}.

\section{Basics}

The category of $\Lambda$-rings has all limits and colimits, and they
agree, as rings, with those taken in the category of rings.
(E.g.~\cite{Borger-Wieland:PA})
We will only need
to take tensor products, intersections, 
and images of morphisms, and it is quite easy
to show their existence on the subcategory of torsion-free
$\Lambda$-rings using the equivalent definition given in the introduction.
For any ring $R$, let $R[\mu_r]$ denote $R[z]/(z^r-1)$.  Because $R[\mu_r]=R\tn\bZ[z]/(z^r-1)$, the ring $R[\mu_r]$ is naturally a $\Lambda$-ring if $R$ is.

Given a $\Lambda$-ring $R$, it will be convenient to call a
$\Lambda$-ring $K$ equipped with a map $R\to K$ of $\Lambda$-rings an
$R\Lambda$-ring. (Compare~\cite[1.13]{Borger-Wieland:PA}.) 
When we say $K$ is flat, or
\'etale, or so on, we mean as $R$-algebras in the usual sense.
We call a sub-$\Lambda$-ring of a $\bQ\Lambda$-ring a $\Lambda$-order if it is finite over $\bZ$.  We do not require that it have full rank.

\begin{prop}
Let $K$ be a finite \'etale $\bQ\Lambda$-ring.  Then $K$ has a
$\Lambda$-order that contains all others.
\end{prop}

We call this $\Lambda$-order the maximal $\Lambda$-order of $K$.

\begin{proof}
Because any $\Lambda$-order $A$ is contained in the usual maximal order of $K$,
which is finite over $\bZ$, it
is enough to show that any two $\Lambda$-orders $A$ and $B$ are contained in a
third.  But $A\tn B$ is a $\Lambda$-ring that is finite over $\bZ$.
Since $A\tn B$ is the coproduct in the category of $\Lambda$-rings, the 
map $A\tn B \to K$ coming from the universal property of coproducts
(i.e., $a\tn b\mapsto ab$) is a $\Lambda$-ring map.
Therefore its image is a $\Lambda$-ring that is finite over $\bZ$, is contained in $K$, 
and contains $A$ and $B$.
\end{proof}

\begin{prop}
\lb{prop-intersection}
Let $K\subseteq L$ be an inclusion of finite \'etale
$\bQ\Lambda$-rings.  Let $A\sub K$ and $B\sub L$ be their maximal
$\Lambda$-orders.  Then $A=K\cap B$.
\end{prop}
\begin{proof}
The intersection 
$K\cap B$ is on the one hand a sub-$\Lambda$-ring of $K$ and, on the other,
finite over $\bZ$.  It is maximal among such rings because of the
maximality of~$B$.
\end{proof}

%Let $r>0$ be an integer, and let $\bZ[\mu_r] = \bZ[x]/(x^r-1)$ and
%$\bQ[\mu_r] = \bQ[x]/(x^r-1)$.  We will always view them as
%$\Lambda$-rings by taking $\psi_p(x)=x^p$ for every prime number $p$.

We can now prove the sufficiency of the conditions of theorem~\ref{thm-A}.
For any ring $R$, 
let $R^\circ$ denote $R$ itself but viewed only as a monoid under multiplication.
So the group $R^*$ of units is just the group of invertible elements 
of the monoid $R^\circ$.

\begin{prop}
Let $r$ be a positive integer, 
let $S$ be a finite $\bm{r}$-set, and let $K$ be the corresponding finite
\'etale $\bQ\Lambda$-ring.
%, viewing $S$ as a $G_{\bQ}$-set
%through the cyclotomic character $G_{\bQ}\to\bg{r}$.  
Then $K$ has an integral $\Lambda$-model.
\end{prop}

\begin{proof}
Take a set $T$ (such as $S$) admitting a surjection 
$\coprod_T\bm{r}\to S$ of $\bm{r}$-sets, the left side denoting
the free $\bm{r}$-set generated by $T$.   Let $L$ be the corresponding finite \'etale $\bQ\Lambda$-ring.
Then $K$ is naturally a sub-$\Lambda$-ring of $L$.  On
the other hand $L$ is $\bQ[\mu_r]^T$ and so has a $\Lambda$-model
$\bZ[\mu_r]^T$.  The intersection of this with $K$ is then
both a $\Lambda$-ring and an order of full rank in $K$.
\end{proof}

\section{Necessary conditions}

Let $K$ be a finite \'etale
$\bQ\Lambda$-ring admitting an integral $\Lambda$-model $A$, 
and let $S=\Hom(K,\bar{\bQ})$
be the corresponding $G_\bQ\times\zpos$-set.

The purpose of this
section is to show that there is an integer $r>0$ such that this
action factors through the map $G_\bQ\times\zpos\to\bm{r}$
%, where $\bm{r}$ denotes the monoid of integers modulo $r$ under multiplication, 
given by the cyclotomic character on the first factor and reduction modulo $r$
on the second.

As usual, we say a
prime number $p$ is unramified in $A$ if $A/pA$ has no non-zero nilpotent elements.
A prime is ramified in $A$ if and only if it divides the discriminant of $A$.
Therefore the set of primes that ramify in $A$ is finite and contains
the set of primes that ramify in the usual maximal order of $K$.

\begin{prop}
\lb{prop-frob-auto}
The endomorphism $\psi_p$ of $A$ is an automorphism if and only if $p$
is unramified in $A$.  In this case, $\psi_p$ is the unique lift of the Frobenius endomorphism of $A/pA$.
% satisfies $\psi_p(K_i)=K_i$ for all $i$.
\end{prop}
\begin{proof}
If $\psi_p$ is an automorphism, then the Frobenius endomorphism $x\mapsto x^p$ of
$A/pA$ is an automorphism, and so $p$ is unramified.

Suppose instead that $p$ is unramified. Then $A/pA$ is a finite product of finite
fields, and so
the Frobenius endomorphism of $A/pA$ is an automorphism of finite order.
The category of finite \'etale $\bZ_p$-algebras is equivalent to the category of 
finite \'etale $\bF_p$-algebras, by way of the functor $\bF_p\tn_{\bZ_p}\vbl$.
(See~\cite[IV (18.3.3)]{EGA-no.32}, say).
%a form of Hensel's lemma.  
Thus the endomorphism 
$1\tn\psi_p$ of $\bZ_p\tn A$ is the unique Frobenius lift and it is an
automorphism of finite order.
%By lifting idempotents with Hensel's lemma one sees that $\bZ_p\tn A$ is a product
%of unramified extensions of $\bZ_p$, which has a unique Frobenius lift.
It follows that $\psi_p\:A\to A$ is the unique Frobenius lift to $A$, and is
also an automorphism.
\end{proof}

\begin{prop}
  \lb{prop-cft} 
  There is a positive integer $c$, divisible only by
  primes that ramify in $A$, such that the action
  of $G_\bQ$ on $S$ factors through the cyclotomic character
  $G_{\bQ}\to\bg{c}$. If $p$ is unramified in $A$, then 
  $p\in\zpos $ and $(p\bmod c)\in\bg{c}$ act in the same way on $S$.
\end{prop}
\begin{proof}

Define the number field $N$ to be the invariant field of the kernel of the map
$G_\bQ \to \Map(S,S)$. Write $\bar{G}=\Gal(N/\bQ)$ and let
${\mc O}_N$ be the ring of integers of $N$.

Take any element $g\in\bar{G}$. 
By Chebotarev's theorem~\cite[V.6]{Neukirch:CFT} there is an
unramified prime $\p$ of $N$ lying over a prime number $p$ such that
$g(x)\equiv x^p \bmod \p$ for all $x\in {\mc O}_N$, i.e.,
$g$ is the Frobenius element of $\p$ in the extension $\bQ\subset N$.
Since Chebotarev's theorem provides infinitely many such $\p$,
we may also assume that $A$ is unramified at $p$.

We now claim that for all $s \in S=\Hom(A,{\mc O}_N)$ the maps
$s\circ \psi_p$ and $g\circ s$ from $A$ to ${\mc O}_N$ are equal.
Since $A$ is unramified at $p$, the map
$\Hom(A,{\mc O}_N)\to \Hom(A,{\mc O}_N/\p)$ is injective, so
it suffices to show that their
compositions with the map ${\mc O}_N \to {\mc O}_N/\p$ are equal.
But this follows from \ref{prop-frob-auto} and our choice of~$\p$.
Thus, $g\in \bar{G}$ and $p\in \zpos$ act in the same way on $S$.

It follows that the image of $G_\bQ$ in $\Map(S,S)$ is contained in the image
of $\zpos$, so $\bar{G}$ is abelian.  By the Kronecker-Weber
theorem~\cite[III.3.8]{Neukirch:CFT}, $N$ is contained in a cyclotomic field
$\bQ(\mu_c)$, where $c$ is divisible only by primes that ramify
in $N$. Since $N$ is the
common Galois closure of the components of $A\otimes\bQ$, such primes are
ramified in $A$ as well. The last statement follows from the fact that for any
prime number $p\nmid c$ the element $(p\bmod c)\in\bg{c}$ corresponds to the
Frobenius element of any prime over $p$ in the extension $\bQ\subset\bQ(\mu_c)$.
\end{proof}

It follows that our map of topological monoids $G_\bQ\times\zpos \to \Map(S,S)$
factors through $\zhat^*\times\zpos$. We will show that it
factors further through $\zhat^{\circ}$ with the following criterion.

\begin{prop}
\label{crit}
A continuous action of $\zhat^*\times\zpos$ on a finite discrete set
$T$ factors through a continuous action of $\zhat^{\circ}$ if and only if
\renewcommand{\theenumi}{\roman{enumi}}
\begin{enumerate}
\item all but finitely many primes $p\in\zpos$ act as automorphisms on $T$, and
\item for all $d\in\zpos$ there exists an integer $c_d$ such that 
%	\begin{enumerate}
%		\item[---] 
			the action of $\zhat^*$ on $dT$ factors through $\bg{c_d}$ and
%		\item[---] 
			for each $n\in\zpos$ with $ndT=dT$ we have 
		\begin{enumerate}
			\item[---] $n$ is relatively prime to $c_d$, and
				%the element $(n\bmod c_d)\in\bZ/c_d\bZ$ 
				%lies in $\bg{c_d}$, and
			\item[---] the elements $(n\bmod c_d)\in\bg{c_d}$ and $n\in\zpos$
				act on $dT$ in the same way.
		\end{enumerate}
%	\end{enumerate}
\end{enumerate}
\end{prop}
%In fact, the argument below establishes an effective version of this proposition.

%gives gives $c$ effectively.
%, though we will leave its formulation to the reader.
%Also, it is possible to give another equivalent characterization by replacing
%(ii) with a condition that is weaker, in the absence 
%of (i), but we do not need it.
%
\begin{proof}
To show the necessity of (i) and (ii), assume that the action of 
$\zhat^*\times\zpos$
factors through $\bm{r}$ for some integer $r>0$.  
Then all primes not dividing $r$, when viewed as elements of
$\zpos$, act as automorphisms on $T$. 
This establishes (i).  
To show (ii), take any $d\in\zpos$ and let $c_d$
be the smallest positive integer $c$ for which 
the $\zhat^*$-action on $dT$ factors through $\bg{c}$.
Note that $c_d$ divides all $c$ with this property.
Now suppose that $p$ is a prime with $pdT=dT$.
Write $r=p^ne$, with $p\nmid e$.
Then for any $x,y\in\bm{r}$ with $x\equiv y \bmod e$ and any $s\in dT$
(in fact any $s\in T$), we have
\[
p^n\big(xs\big) = (p^nx)s = (p^ny)s = p^n(ys). 
\]
Since $p$ acts bijectively on $dT$, this implies that $x$ and $y$ act in the
same way on $dT$, so the action of $\zhat^{\circ}$ on $dT$ factors through
$\bm{e}$. In particular, $c_d\mid e$, so $p\nmid c_d$, and the elements
$p\in \zpos$ and $(p\bmod e)\in\bg{e}$ and $(p\bmod c_d)\in\bg{c_d}$ all
act in the same way on $dT$. Since $\zpos$ is generated by the primes,
part (ii) follows.

For the converse, suppose (i) and (ii) hold.  
For every prime number $p$, let $a_p$ be the smallest integer $a\ge 0$
such that $p^aT = p^{a+1}T$. By (i) we have $a_p=0$ for all but finitely many $p$,
so $r_0=\prod_p p^{a_p}$ is an integer. 
Note that for any $n\in\zpos$ we have $nT=\gcd(n,r_0)T$.

Now let $r$ be any integer divisible by $d c_d$ for every $d\dvd r_0$.
We will show that the action of $\zhat^*\times \zpos$ on $T$ factors through 
$\bm{r}$.  
To do this, we will show directly that
any two elements $(a_1,d_1),(a_2,d_2)\in\zhat^*\times\zpos$ satisfying
$a_1d_1 \equiv a_2d_2 \bmod r$ act in the same way on $T$.
%We will show that $(a_1,d_1)$ and
%$(a_2,d_2)$ act in the same way on $T$.

Since $r_0\mid r$ the congruence implies that $d_1$ and
$d_2$ have the same greatest common divisor
$d$ with $r_0$, so we have $d_1T=dT=d_2T$.
For $i=1, 2$ we define $d'_i\in \zpos$ by $d_i=dd'_i$ and deduce that
$d'_i(dT)=dT$.  Using (ii) one sees that $d'_i$ is coprime to $c_d$, and
that the action of $d'_i$ on $dT$ is that of $(d'_i\bmod c_d)\in\bg{c_d}$.
By the defining property of $r$ we have $dc_d\mid r$, so $a_1d'_1d \equiv
a_2d'_2d \bmod d c_d$, which implies $a_1d'_1\equiv a_2d'_2 \bmod c_d$.
It follows that $(a_1,d'_1)$ and $(a_2,d'_2)$ are mapped to the same element
of $\bm{c_d}$, which in fact lies in $\bg{c_d}$.
Thus, $(a_1,d'_1)$ and $(a_2,d'_2)$ act identically on $dT$, and composing with
$(1,d)$ we see that $(a_1,d_1)$ and $(a_2,d_2)$ act in the same way on $T$.
\end{proof}

In order to finish the proof of theorem \ref{thm-A} one checks the conditions
of \ref{crit} for $T=S$.  Condition (i) follows from \ref{prop-frob-auto} and the
fact that $A$ is ramified at only finitely many primes. For condition (ii), suppose
$d\in \zpos$ is given and consider the sub-$\Lambda$-ring $\psi_d(A)$ of
$A$ which corresponds to the $G_\bQ\times\zpos$-set $dS$ of $S$.  Proposition
\ref{prop-cft} applied to $\psi_d(A)$ now provides an integer $c_d$ so that the
$G_\bQ$-action on $dS$ factors through $\bg{c_d}$.  Any $n\in\zpos$ with
$ndS=dS$ is a product of primes that are unramified in $\psi_d(A)$ by
\ref{prop-frob-auto}, and so \ref{prop-cft} tells us that
$(n\mod c_d)\in\bg{c_d}$ and that this element acts on $dS$ in the same way 
as $n$. This gives condition (ii).

\section{Explicit maximal $\Lambda$-orders}

Given a prime number $p$, there is a notion of $\Lambda_p$-action on a
ring $R$, and as before, this has a simple description if $R$ has no
$p$-torsion: a ring endomorphism $\psi_p$ of $R$ that lifts the
Frobenius endomorphism, that is, such that $\psi_p(x)-x^p\in pR$ for
all $x\in R$.  Also as before, a sub-$\Lambda_p$-ring $A$ of a $\bQ_p\Lambda_p$-ring
$K$ is called
a $\Lambda_p$-order if it is finite over $\bZ_p$. It is said to be
maximal if it contains every other $\Lambda_p$-order in~$K$.
%It is called a $\Lambda_p$-model of $K$ if it is finite over $\bZ_p$ and $\bQ\tn A=K$.

We have two natural ways of making $\Lambda_p$-orders. First,
for any abelian group $V$ the group ring $\bZ_p[V]$ is a
$\Lambda_p$-ring when we set $\psi_p(r)=r$ for $r\in \bZ_p$ and
$\psi_p(v)=v^p$ for $v\in V$.
Secondly, if $A$ is  the ring of integers of a finite unramified extension $K$
of $\bQ_p$, then $A$ has a unique $\Lambda_p$-ring structure, where
$\psi_p$ is the Frobenius map (cf. \ref{prop-frob-auto}). By extending
$\psi_p$ to $K$ we see that $A$ is the maximal $\Lambda_p$-order of $K$.
Taking tensor products of these two building blocks we see that for
any integer $q$ the ring $A[\mu_q]=A[z]/(z^q-1)$ is a
$\Lambda_p$-order in $K[\mu_q]$.

\begin{lemma}
\lb{lemma-local}
If $q$ is a power of $p$, then $A[\mu_q]$ is the maximal
$\Lambda_p$-order of $K[\mu_q]$.
\end{lemma}
\begin{proof}
%First note that $\psi_p(z)=z^p$ 
By induction, it is enough to assume $A[\mu_q]$ is maximal and then
prove $A[\mu_{pq}]$ is.

Let $k$ denote the residue field of $A$, and
let $\zeta$ denote a primitive $pq$-th root of unity in some extension
of $K$.  Then we have $K[z]/(z^{pq}-1) = K(\zeta) \times
K[y]/(y^q-1)$, the element $z$ corresponding to $(\zeta,y)$.  In
these terms, the $\Lambda_p$-action is given by $\psi_p(b,f(y)) =
(f^*(\zeta^p), f^*(y^p))$, where $f^*(y)$ denotes polynomial 
obtained by applying the Frobenius map $\psi_p$ coefficient-wise to $f(y)$.

%Now consider the following diagram of rings:
%\[
%\xymatrix{
%A[z]/(z^{pq}-1) \ar^{z\mapsto y}@{>>}[r]\ar^{z\mapsto\zeta}@{>>}[d] 
%  & A[y]/(y^q-1)\ar^{y\mapsto y}@{>>}[d] \\
%A[\zeta] \ar^-{\zeta\mapsto y}@{>>}[r]& k[y]/(y^q-1),
%}
%\]
%where $A[\zeta]$ denotes the ring of integers in $K(\zeta)$.
%Note that the bottom map exists because of the following equality of
%ideals in $A[z]/(z^{pq}-1)$:
%\[
%(1+z^q+\cdots+z^{q(p-1)}, z^q-1) = (z^q-1,p).
%\]
%This also implies the kernel of the map is generated by
%$\zeta_p-1$, where $\zeta_p$ is a primitive $p$-th root of unity.
%
%We will now show this is a pull-back diagram.  (See also 
%Kervaire--Murthy~\cite{Kervaire-Murthy}.)  We know that
%$A[z]/(z^{pq}-1)$ is a subring of the pull-back because it is
%torsion-free and the diagram is a pull-back diagram after tensoring
%with $K$.  Because the kernel of the map on the left side is generated
%by $1+z^q+\cdots+z^{q(p-1)}$, it surjects onto the kernel $(p)$ of the
%map on the right side.  A diagram chase then shows $A[z]/(z^{pq}-1)$
%is the entire pull-back.
%

Now consider the following diagram of rings:
\[
\xymatrix{
A[z]/(z^{pq}-1) \ar^{z\mapsto y}@{>>}[r]\ar^{z\mapsto\zeta}@{>>}[d] 
  & A[y]/(y^q-1)\ar^{y\mapsto y}@{>>}[d] \\
A[\zeta] \ar^-{\zeta\mapsto y}@{>>}[r]& k[y]/(y^q-1),
}
\]
where $A[\zeta]$ denotes the ring of integers in $K(\zeta)$.
As noted for example in Kervaire--Murthy~\cite{Kervaire-Murthy}, this is a pull-back 
diagram.  This is just an instance of the easy fact that
for ideals $I$ and $J$ in any ring $R$, we have
	\[
	R/(I\cap J) = R/I \times_{R/(I+J)} R/J.
	\]
In our case, take $R=A[z]$, $I=(z^q-1)$, and $J=(1+z^q+\cdots+z^{q(p-1)})$.

Now suppose $R$ is a $\Lambda_p$-order in $K[\mu_{pq}]$.  Then the image 
of $R$ in $K[y]/(y^q-1)$
is contained, by induction, in $A[y]/(y^q-1)$.  Therefore $R$ is
contained in $A[\zeta]\times A[y]/(y^q-1)$; we will view elements of
$R$ as elements of this product without further comment.  
Because the diagram above is a pull-back diagram, we need only show that the two maps $R\rightrightarrows k[y]/(y^q-1)$ given by mapping to the two factors and then projecting to $k[y]/(y^q-1)$ agree.

%Because the diagram above is a pull-back diagram, to show $R$ is contained in
% $A[z]/(z^{pq}-1)$, it is enough to show that the projections onto the factors
%  $A[\zeta]$ and $A[y]/(y^q-1)$ agree after further projecting to $k[y]/(y^q-1)$.  
%  This is the same as showing $\zeta_p-1$ generates the ideal $I$ in $A[\zeta]$ which 
%  is the pre-image of the ideal in $k[y]/(y^q-1)$ generated by the differences of these
%   two projections.  And $I$ is the same as the ideal generated by $\zeta_p-1$ and the
%    elements $b-f(\zeta)$ for all $(b,f(y))\in R$.  (Note that $f(\zeta)$ is only 
%    well-defined modulo $\zeta_p-1$.)

  Let $v$ denote the valuation on $A[\zeta]$ normalized so that $v(p)=1$.
  Let $a\leq 1/(p-1)$ be the largest number such that 
  \begin{equation}
  \label{aequ}
  \text{for all } (b,f(y))\in R \text{ we have } v(b-f(\zeta))\geq a.
  \end{equation}
  Note that the expression $f(\zeta)$ makes sense only modulo
  $\zeta_p-1$, and so the condition above is meaningless if
  $a>v(\zeta_p-1)=1/(p-1)$. 

 Let $(b,f(y))$ be an element of $R$.
 Then because $R$ is
 a $\Lambda_p$-ring, there is an element $(c,g(y))\in pR$ such that
 \[
 (c,g(y)) = (b,f(y))^p - \psi_p(b,f(y)) = (b^p-f^*(\zeta^p), f(y)^p-f^*(y^p)).
 \]
 On the other hand, because of our assumption on $a$, we have
 \[
 1+a \leq v(c-g(\zeta)) = 
	v\left((b^p-f^*(\zeta^p))-(f(\zeta)^p-f^*(\zeta^p)\right) =
	v(b^p-f(\zeta)^p). 
 \]
 But because the integral polynomial $p(X-Y)$ divides
 $(X-Y)^p-(X^p-Y^p)$, we have 
 \[
 v((b-f(\zeta))^p-(b^p-f(\zeta)^p)) \geq 1+a
 \]
 and hence $v(b-f(\zeta))\geq (1+a)/p.$  That is, $(1+a)/p$
 satisfies~(\ref{aequ}). 
 
 But then $a=1/(p-1)$ because otherwise this would violate the
 maximality of $a$.  
 In other words, for any element $(b,f(y))\in R$,
 the element $\zeta_p-1$ divides $b-f(\zeta)$.  
 This is just another way
 of saying $b$ and $f(y)$ have the same image in
 $A[\zeta]/(\zeta_p-1)=k[y]/(y^q-1)$, and hence the element $(b,f(y))$
 lies in the fiber product, which we showed above is
 $A[\mu_{pq}]$.
\end{proof}

\subsection{}
{\em Remark.}
It is not true that $\bZ_p[V]$ is maximal for every finite abelian
group $V$.  For example, if $V=\bZ/p\bZ \times \bZ/p\bZ$ and
$x=\frac{1}{p}\sum_{\sigma\in V}\sigma$, then $\psi_p(x)=p$ and
$x^2=px$.  Therefore, the subring $\bZ_p[V][x]$ of $\bQ_p[V]$ is a
sub-$\Lambda_p$-ring which is finite over $\bZ_p$.  The global analogue also 
holds: $\bZ[V]$ is not the maximal integral $\Lambda$-model for $\bQ[V]$.  This 
follows from the local statement.

\begin{prop}
\lb{prop-local-to-global}
Let $K$ be a finite \'etale $\bQ\Lambda$-ring, and let
$R$ be a $\Lambda$-order.  If $\bZ_p\tn R$ is
a maximal $\Lambda_p$-order in $\bQ_p\tn R$ for all primes $p$, then
$R$ is maximal.
\end{prop}

%The converse, however, is not true. 

\begin{proof}
Let $S$ be the maximal $\Lambda$-order in $\bQ\tn R$.  Because
$\bZ_p\tn R$ is the maximal $\Lambda_p$-order in $\bQ_p\tn R$, the inclusion
$\bZ_p\tn R \subseteq \bZ_p\tn S$ is an equality.  Therefore $R$ and $S$ have the same rank,
and $p$ does not divide the index of $R$ in $S$.
\end{proof}

\begin{thm}
Let $r\geq 1$ be an integer.  Then $\bZ[\mu_r]$ is the maximal
$\Lambda$-order in $\bQ[\mu_r]$.
\end{thm}
\begin{proof}
By \ref{prop-local-to-global}, it suffices to show that for every prime
$p$, the $\Lambda_p$-order $\bZ_p[\mu_r]$ maximal.
Write $r=qn$, where $q$ is the largest power of $p$ dividing $r$.
Then $\bZ_p[\mu_r] = \bZ_p[\mu_n][\mu_q] = \prod_A A[\mu_q]$, where
$A$ runs over the irreducible factors of $\bZ_p[\mu_n]$, all of which
are unramified extensions of $\bZ_p$.  By \ref{lemma-local}, 
each factor $A[\mu_q]$ is a maximal $\Lambda_p$-order, and therefore
so is their product.  
%Finally,~\ref{prop-local-to-global} implies
%$\bZ[\mu_r]$ is a maximal $\Lambda_p$-order.
\end{proof}

%\begin{cor}
%Let $S$ be a finite $\bm{r}$-set.  Then the maximal
%$\Lambda$-order in the corresponding finite
%\'etale $\bQ$-algebra, $\Hom_{\bg{r}}(S,\bQ(\zeta_r))$, 
%consists in the subset of elements $f$ such that for all $s\in S$
%\end{cor}

\subsection{}
{\em Remark.}
Using \ref{cor-C} and \ref{prop-intersection}, we can also describe the maximal $\Lambda$-order in general as follows.
Let $S$ be a finite $\bm{r}$-set, 
let $\zeta_r\in\bar\bQ$ denote a primitive $r$-th root of unity,
and let
$K=\Hom_{\bg{r}}(S,\bQ(\zeta_r))$
denote the corresponding finite \'etale $\Lambda$-ring over $\bQ$.  
%writing $\bQ[\mu_r]=\prof_{d\dvd r}\bQ(\zeta_d)$, 
Consider the isomorphism
\[
\bQ[\mu_r] \to \prod_{d\dvd r} \bQ(\zeta_{r}^d)
\]
given by $z\mapsto (\dots,\zeta_r^d,\dots)_{d\dvd r}$.
Then the maximal
$\Lambda$-order in $K$ is the set of elements $f\in K$ such that for all
$s\in S$, the element 
\[
(\dots,f(ds),\dots)_{d\dvd r}\in \prod_{d\dvd r} \bQ(\zeta_r^d) \cong \bQ[\mu_r]
\]
lies in $\bZ[\mu_r]$.

\bibliography{references}
\bibliographystyle{plain}
\end{document}